\documentclass[reqno,a4paper,12pt]{amsart} 

\usepackage{amsmath,amscd,amsfonts,amssymb}
\usepackage{mathrsfs,dsfont}

\usepackage{tikz}

\numberwithin{equation}{section}
\numberwithin{figure}{section}

\def\R{\mathbb{R}}

\def\Z{\mathbb{Z}}

\def\E{\mathscr{E}}
\def\F{\mathscr{F}}

\def\P{\mathscr{P}}

\def\lam{\lambda}

\def\1{\mathds{1}}
\def\eps{\varepsilon}

\renewcommand\leq{\leqslant}
\renewcommand\geq{\geqslant}

\newcommand{\ft}[1]{\widehat #1}
\newcommand{\dotprod}[2]{\langle #1 , #2 \rangle}

\newcommand{\vol}{\operatorname{vol}}
\newcommand{\Vol}{\operatorname{Vol}}

\theoremstyle{plain}
\newtheorem{thm}{Theorem}[section]
\newtheorem{lem}[thm]{Lemma}
\newtheorem{corollary}[thm]{Corollary}
\newtheorem{prop}[thm]{Proposition}

\newtheorem*{claim*}{Claim}

\newcommand{\thmref}[1]{Theorem~\ref{#1}}
\newcommand{\secref}[1]{Section~\ref{#1}}

\newcommand{\lemref}[1]{Lemma~\ref{#1}}

\newcommand{\propref}[1]{Proposition~\ref{#1}}

\newcommand{\corref}[1]{Corollary~\ref{#1}}

\theoremstyle{definition}

\newtheorem*{definition*}{Definition}
\newtheorem*{remarks*}{Remarks}
\newtheorem*{remark*}{Remark}
\newtheorem{remark}[thm]{Remark}

\newenvironment{enumerate-math}
{\begin{enumerate}
\addtolength{\itemsep}{5pt}
}
{\end{enumerate}}

\newenvironment{enumerate-text}
{\begin{enumerate}
\addtolength{\itemsep}{5pt}
}
{\end{enumerate}}

\begin{document}

 \title[Multi-tiling and equidecomposability of polytopes]
{Multi-tiling and equidecomposability of polytopes by lattice translates}

\author{Nir Lev}
\address{Department of Mathematics, Bar-Ilan University, Ramat-Gan 5290002, Israel}
\email{levnir@math.biu.ac.il}

\author{Bochen Liu}
\address{Department of Mathematics, Bar-Ilan University, Ramat-Gan 5290002, Israel}
\email{bochen.liu1989@gmail.com}

\thanks{Research supported by ISF grant No.\ 227/17 and ERC Starting Grant No.\ 713927.}
\date{July 29, 2019}


\begin{abstract}
	We characterize the polytopes in $\R^d$ (not necessarily convex or connected ones)
	which multi-tile the space by translations along a given lattice. 
	We also give a necessary and sufficient condition for two polytopes in 
	$\R^d$ to be equidecomposable by lattice translations.
\end{abstract}

\maketitle

% =========================================================

\section{Introduction} \label{secI1}

A \emph{simplex} in $\R^d$ is the convex hull of $d + 1$ points
 which do not all lie in some hyperplane. By a \emph{polytope} in $\R^d$ we
mean a set which
can be represented as the union of a finite number of simplices
with disjoint interiors. Remark that a polytope is not necessarily
a convex, nor even a connected, set.

By a \emph{lattice} in $\R^d$ we  refer to a discrete subgroup 
generated by $d$ linearly independent vectors.

\subsection{}
\label{secI1.A}
Let $A$ be a polytope in $\R^d$, and $L$ be a lattice in $\R^d$. 
The polytope $A$ is said to \emph{$k$-tile} by translations with respect to
$L$ (where $k$ is a positive integer) if almost every point of the space $\R^d$
 lies in exactly $k$ among the sets $A+\lambda$ $(\lam \in L)$, that is,
\begin{equation}
\label{eq:ztiling}
\sum_{\lam \in L} \chi_A(x - \lam) = k \quad \text{a.e.}
\end{equation}
where $\chi_A$ is the indicator function of $A$.
In this case, the number $k$ is called the \emph{level}, or the \emph{multiplicity}, of the tiling.

It is well-known that the
condition $\vol(A)  = k \det(L)$ is  necessary for  $k$-tiling
by translations of $A$ along $L$, where $\vol(A)$ denotes the volume
of $A$, and $\det(L)$ is the volume of some (or equivalently, of any) fundamental parallelepiped
of the lattice $L$.

We are concerned  with the following problem:

\emph{Given a polytope $A$ and a lattice $L$, formulate in effective 
terms a condition which is necessary and sufficient for the
translates of $A$ along $L$ to be a $k$-tiling.}

This problem was studied for some special classes of polytopes,
by several authors. For example, if $A$ is a convex polygon in two dimensions, 
then a necessary and sufficient condition for $k$-tiling was given by Bolle \cite{Bol94}. 
Kolountzakis \cite{Kol00} extended this result for a wider class of planar polygons,
which includes also non-convex ones. Gravin, Robins and Shiryaev proved in \cite{GRS12} 
that for a polytope $A \subset \R^d$  to $k$-tile by translations along a lattice $L$,
it is sufficient that $A$ be centrally symmetric, have centrally symmetric facets, 
and that all the vertices of $A$ lie in $L$. It was also proved in 
 \cite{GRS12} that if $A$ is convex, then 
the first two conditions, namely that $A$ be centrally symmetric and have centrally 
symmetric facets, are also necessary for $k$-tiling.

In this paper, we give a complete characterization of the polytopes $A \subset \R^d$
 (not necessarily convex or connected ones)  which tile at some level $k$ by translations along a lattice $L$.
We formulate this characterization in terms of \emph{Hadwiger functionals}, whose definition
will now be given.

\subsection{}
Let $r$ be an integer, $0 \leq r \leq d-1$, and suppose that
\begin{equation}
\label{eq:subsps}
V_r \subset V_{r+1} \subset \cdots \subset V_{d-1} \subset V_d = \R^d 
\end{equation}
is a sequence of affine subspaces such that $V_j$ has dimension $j$ (where an affine
 subspace means a translated linear subspace). Each subspace $V_j$ ($r \leq j \leq d-1$)
 in the sequence divides the next one $V_{j+1}$ into two half-spaces; we call one of them
 the positive half-space, and the other the negative half-space. Such a sequence of
 nested affine subspaces, endowed with a choice of positive and negative half-spaces, 
will be called an \emph{$r$-flag}, and will be denoted by $\Phi$.

Let $A$ be a polytope in $\R^d$, and suppose that $A$ has a sequence of faces 
\begin{equation}
\label{eq:faces}
F_r \subset F_{r+1} \subset \cdots \subset F_{d-1} \subset F_d = A,
\end{equation}
where $F_j$ is a $j$-dimensional face of $A$ which is contained in $V_j$
($r \leq j \leq d-1$). To each  face $F_j$  we assign a coefficient $\varepsilon_j$, given by
 $\varepsilon_j =+1$ if $F_{j+1}$ adjoins $V_j$ from the positive side, 
and $\varepsilon_j = -1$ if $F_{j+1}$ adjoins $V_j$ from the negative side. 
We then define 
\begin{equation}
\label{eq:weight}
\omega_{\Phi}(A) = \sum \varepsilon_{r} \varepsilon_{r+1} \cdots \varepsilon_{d-1} \Vol_r (F_r) ,
\end{equation}
where the sum goes through all sequences of faces of $A$ with the property above, and 
where $\Vol_r(F_r)$ denotes the $r$-dimensional volume of $F_r$. If no such a
 sequence of faces of $A$ exists, then we let $\omega_{\Phi}(A) = 0$.
In the case when $r=0$, by a $0$-dimensional face we mean a vertex of $A$, 
and its $0$-dimensional volume is defined to be $1$. 

Finally, let $L$ be a lattice in $\R^d$. For each $r$-flag $\Phi$ ($0 \leq r \leq d-1$) we define
\begin{equation}
\label{eq:hadwiger}
H_{\Phi}(A,L) = \sum_{\Psi} \omega_{\Psi}(A) , 
\end{equation}
where $\Psi$ runs through all $r$-flags for which there exists $\lam \in L$ such that $\Psi$
can be obtained from $\Phi$ by translating all the 
affine subspaces in \eqref{eq:subsps}, as well as the positive and negative half-spaces, 
 by the vector $\lam$.
(Remark that each $r$-flag $\Psi$ in the sum is taken into account only once, 
even if there is more than a single vector $\lam \in L$ which carries $\Phi$ onto $\Psi$.)
  Notice that there can be only finitely many nonzero terms in the sum \eqref{eq:hadwiger},
hence this sum has a well-defined value.
We will call $H_{\Phi}$ the \emph{Hadwiger functional} associated to the $r$-flag $\Phi$.

We can now state our result that characterizes the polytopes $A$ whose translates along a lattice $L$
form a $k$-tiling:

\begin{thm}
\label{thmA1.1}
Let $A$ be a polytope in $\R^d$, and $L$ be a lattice in $\R^d$. Then $A$ tiles
at some level $k$ by translations with respect to $L$, if and only if
\begin{equation}
\label{eq:A1.1.1}
H_{\Phi}(A,L) = 0
\end{equation}
for every $r$-flag $\Phi$ $(0 \leq r \leq d-1)$.
\end{thm}

Notice that in order to check the condition in this result,  it is actually enough to 
verify \eqref{eq:A1.1.1} for finitely many flags $\Phi$.\footnote{
Indeed,  it is enough to verify \eqref{eq:A1.1.1} only for flags $\Phi$ 
that can be obtained from some sequence of faces as in \eqref{eq:faces},
by taking each subspace $V_j$ to be the affine hull of the  face $F_j$ $(r \leq j \leq d-1)$.
So the number of verifications needed is bounded by the number of such sequences
of faces.}
Hence, only a finite number of verifications is in fact needed.
Moreover, the value of each $H_{\Phi}(A,L)$ depends on $A$ and $L$
in an elementary way. So the condition in \thmref{thmA1.1} can be effectively checked 
in concrete situations.

In \secref{sec:examples}, we will illustrate the use of \thmref{thmA1.1} 
 by analyzing  some examples. In particular, 
the above mentioned results  from \cite{Bol94}, \cite{Kol00} and \cite{GRS12}
will be recovered based on \thmref{thmA1.1}.

\subsection{}
The $k$-tiling problem is closely related to another subject in discrete geometry, namely,
the theory of equidecomposability of polytopes by rigid motions. This is a classical
subject, which goes back to Hilbert's third problem (one of the famous 23 problems posed by
Hilbert at the International Congress of Mathematicians in 1900).

If $A$ and $B$ are two polytopes in $\R^d$, then they are said to be \emph{equidecomposable}
 (or \emph{dissection equivalent}, or \emph{scissors congruent}) if the polytope $A$
 can be partitioned, up to measure zero, into a finite number of smaller polytopes which can be 
rearranged using rigid motions to form, again up to measure zero, a partition of the polytope $B$.

It is obvious that if two polytopes $A$ and $B$ are equidecomposable, then they must
have the same volume. Hilbert's third problem was concerned with the converse assertion:
are any two  polytopes $A$ and $B$ of the same volume
equidecomposable? It has been  known earlier that in  two dimensions, 
any two polygons of equal area are equidecomposable, but in the same year 1900
it was shown by Dehn that in three dimensions such a result is no longer true (see the book
\cite{Bol78} for a comprehensive exposition).

The problem of equidecomposability of polytopes has also been studied when extra restrictions
are imposed on the way in which the pieces of the partition are allowed to be rearranged. 
If $G$ is a group of rigid motions of $\R^d$, then two polytopes $A$ and $B$ are said to be 
\emph{$G$-equidecomposable} if $A$ can be partitioned into a finite number of polytopes
which can be rearranged, using only motions from $G$, to form a partition of $B$
(where as before, the pieces of each partition may overlap but only up
to measure zero).

In this paper, we consider the following question:

\emph{Given two polytopes $A,B$ and a lattice $L$, under what condition 
$($that can be checked effectively$)$ are $A$ and $B$
equidecomposable using only translations by vectors from $L$?}

Our next result characterizes precisely when such an equidecomposition exists.
To state the result, we again use the Hadwiger functionals $H_{\Phi}$ described above.

\begin{thm}
\label{thmA2.1}
Let $A$ and $B$ be two polytopes in $\R^d$, and let $L$ be a lattice in $\R^d$. Then $A$ and $B$
are equidecomposable with respect to translations by vectors from $L$, if and only if $A$ and $B$ have the same
volume and
\begin{equation}
\label{eq:A2.1.1}
H_{\Phi}(A,L) = H_{\Phi}(B,L) 
\end{equation}
for every $r$-flag $\Phi$ $(0 \leq r \leq d-1)$.
\end{thm}

As was the case with \thmref{thmA1.1}, also the condition in this result is elementary enough 
so that it can be effectively checked in concrete situations.

\subsection{}
In order to understand the role played by the Hadwiger functionals in the theory of equidecomposability 
of polytopes, it is necessary to be acquainted with an important notion in
this theory, namely, the notion of \emph{additive invariants}. 

Let $G$ be a group of rigid motions of $\R^d$. A function $\varphi$, defined on the set of all polytopes in 
$\mathbb R^d$, is said to be an \emph{additive $G$-invariant} if (i) it is additive, namely, if $A$ and $B$
 are two polytopes with disjoint interiors then $\varphi(A \cup B) = \varphi(A) + \varphi(B)$; 
and (ii) it is invariant under motions from the
group $G$, that is, $\varphi(A) = \varphi(g(A))$ whenever $A$ is a polytope and $g \in G$.

It is obvious that for two polytopes $A$ and $B$ to be $G$-equidecomposable, a necessary  
condition is that $\varphi(A) = \varphi(B)$ for any additive $G$-invariant $\varphi$. A general problem 
is to construct a ``complete system'' of additive $G$-invariants, that is, invariants which together provide a 
condition which is both necessary and
sufficient for two polytopes of the same volume to be $G$-equidecomposable.

In his solution to Hilbert's third problem, Dehn constructed an additive invariant with respect to the group of all 
rigid motions of $\R^3$, which allowed him to show that a regular tetrahedron and a cube of the
same volume are not equidecomposable \cite{Deh01}. Dehn invariants for polytopes in $\R^d$ have also been studied
\cite{Had54}, and shown to form a complete system in dimensions $d=3,4$ \cite{Syd65, Jes72}. 
It remains an open problem as to whether these invariants are complete also in dimensions $d \geq 5$.

Additive invariants with respect to the \emph{group of translations} of $\R^d$ were introduced by Hadwiger
\cite{Had52, Had57}. It was proved that these invariants form a complete system, so that 
together they provide a necessary and sufficient condition for two polytopes of the same volume to be 
equidecomposable by translations. This was proved by Hadwiger and Glur in dimension two \cite{HG51}, 
by Hadwiger in dimension three \cite{Had68}, and by Jessen and Thorup \cite{JT78}, and independently 
Sah \cite{Sah79}, in every dimension. 

The question which concerns us in this paper, involves equidecomposability with respect to the group 
of translations by vectors belonging to a lattice $L$. 
How to construct additive invariants with respect to this group of motions?
We can use the fact (observed by Hadwiger)  that if $\Phi$ is an $r$-flag ($0 \leq r \leq d-1$), 
then the  function $\omega_{\Phi}$ defined by
\eqref{eq:weight} is an additive function on the set of all polytopes in $\R^d$ 
(see e.g.\ \cite[Section 2.6]{Sah79}). Hence the Hadwiger 
functional $H_\Phi(\,\cdot\, , L)$ defined by \eqref{eq:hadwiger} is easily seen to be
an additive invariant with respect to translations by vectors from $L$.

In fact, this construction is valid for any subgroup $L$ of $\R^d$ (that is, not only for 
a lattice). In particular, if $L = \R^d$, then the obtained invariants are the classical ones
introduced by Hadwiger. In the case when $L$ is a proper subgroup of $\R^d$, the
invariants $H_\Phi(\,\cdot\, , L)$ were first considered in \cite[Section 5]{GL15}.

This clarifies why the ``only if'' part of \thmref{thmA2.1} is true. Indeed,
for two polytopes 
$A$ and $B$ to be equidecomposable using translations by vectors from $L$, it is
necessary  that $\varphi(A) = \varphi(B)$ 
for any additive function $\varphi$ which is invariant with respect to translations 
by vectors from $L$ (and in particular,
$A$ and $B$ must have the same volume).
 In fact, the same is true 
for any subgroup $L$ of $\R^d$, not necessarily a lattice.

It is therefore the ``if'' part of \thmref{thmA2.1} which constitutes the main result.
The result can be equivalently stated by saying that if $L$ is a lattice, then  the
Hadwiger functionals  $H_\Phi(\,\cdot\, , L)$ form a  \emph{complete system}  of invariants
with respect to the group of translations by vectors from $L$. 
Our proof of this result is based on completely different considerations than the proofs in 
\cite{HG51, Had68, JT78, Sah79}, where it is shown that the classical Hadwiger 
invariants $H_\Phi(\,\cdot\, , \R^d)$ are complete with respect to the group of all the translations  of $\R^d$.

The question as to whether the Hadwiger invariants $H_\Phi(\,\cdot\, , L)$
are complete when $L$ is a proper subgroup of $\R^d$, was raised in \cite[Section 7.1]{GL15}.
\thmref{thmA2.1} provides an affirmative answer to this question
 in the case when  $L$ is a lattice. For a general subgroup $L$, the problem  remains open.

\subsection{}
In all that has been said so far, we have still not explained what is
the relation between the two problems that have been discussed,
that is, the $k$-tiling problem and the equidecomposability problem, both considered
with respect to lattice translates.

 The following elementary fact demonstrates the
relation between these two problems:

\begin{prop}
\label{propA2.2}
Let $A$ and $B$ be two polytopes in $\R^d$, and let $L$ be a lattice in $\R^d$.
Assume that $B$ tiles at some level $k$ by translations along $L$. Then the translates
of $A$ along $L$ is also a $k$-tiling, if and only if $A$ and $B$  are equidecomposable with respect 
to translations by vectors from $L$.
\end{prop}

An interesting case to which this proposition applies is when $B$ is the union of 
$k$ disjoint  fundamental parallelepipeds of the lattice $L$. It is obvious that such
a polytope $B$ tiles at level $k$ by translations along $L$. It is also
easy to check that $H_{\Phi}(B,L) = 0$ for every $\Phi$. Hence, one can see
that \thmref{thmA2.1} and \propref{propA2.2} imply the ``only if'' part of
\thmref{thmA1.1}, namely,  for a polytope $A$
to $k$-tile by translations along a lattice $L$, it is  necessary  that $H_{\Phi}(A,L) = 0$ 
for every $\Phi$.

On the other hand, the ``if'' part of \thmref{thmA1.1} does not follow 
from  \thmref{thmA2.1} and \propref{propA2.2} in a similar way. 
The missing  ingredient which is needed for this purpose, is a result which states
that the vanishing of all the Hadwiger functionals $H_{\Phi}(A,L)$
implies that the volume of $A$ must be an integral multiple of $\det(L)$.

In order to prove Theorems \ref{thmA1.1} and \ref{thmA2.1}, it is
natural to realize the set of all polytopes as a subset 
of a certain  group, called the \emph{polytope group} in $\R^d$.
We will prove a theorem, which is in fact  the main result 
of this paper  (\thmref{thmC3.1}), that characterizes the elements of the 
polytope group  in $\R^d$ whose translates along a lattice $L$ form a $k$-tiling.
Theorems \ref{thmA1.1} and \ref{thmA2.1} will be obtained as
consequences of this result.

% =========================================================

\section{Examples}
\label{sec:examples}

In this section we consider some specific classes of polytopes, and obtain
conditions for the polytopes in these classes to $k$-tile by 
lattice translates. The examples include finite unions of intervals in 
dimension one, convex polytopes in dimensions two and three,
and some other classes.

In analyzing these examples, we will  illustrate how to use the characterization
given in \thmref{thmA1.1} of the polytopes  which $k$-tile by lattice translates. 
In particular we will recover the results  from 
\cite{Bol94}, \cite{Kol00} and \cite{GRS12} mentioned above.

Similar examples were analyzed in \cite[Section 5]{GL15}, where the Hadwiger 
functionals were studied not with respect to lattice translations, but with respect to 
translations along a countable, dense subgroup of $\R^d$ generated by $d+1$ vectors.
Some of the results of this section are analogous to those discussed in \cite{GL15}.

\subsection{}
The simplest situation to which \thmref{thmA1.1} applies, takes place in dimension one, where
a ``polytope'' $A$ is a finite union of disjoint closed intervals  $[ a_j , b_j ]$, and a ``lattice''
$L$ is an arithmetic progression passing through the origin.

In dimension one, a  $0$-flag $\Phi$ is determined by a single point $x$, that divides $\R$
 into a positive and a negative half-line. The value of the Hadwiger functional 
$H_\Phi(A,L)$ is then equal to the difference between the number of left endpoints $a_j$,
 and the number of right endpoints $b_j$, which belong to the set $L + x$.
 Hence, it follows from  \thmref{thmA1.1} that  for $A$ to tile at some level $k$ by translations
with respect to $L$, it is necessary and sufficient that any translate of $L$ contain the same 
number of left and right endpoints. 

This condition can be equivalently stated  as follows:

\begin{corollary}
\label{corA1.2}
Let $A \subset \R$ be the union of $n$ disjoint intervals $[a_j, b_j]$, $1 \leq j \leq n$,
and let $L$ be a lattice in $\R$. Then $A$ tiles at some level $k$ by translations along
$L$, if and only if there exists a permutation $\sigma$ of $\{1,2, \ldots , n\}$ 
such that $b_{\sigma (j)} - a_j \in L$ for every $j$.
\end{corollary}

\subsection{}
In our next example, we consider the case when $A$ is a two-dimensional convex polygon. In this
case, we will see that \thmref{thmA1.1} recovers the main result from \cite{Bol94}. The result
can be stated the following way:

\begin{corollary}[Bolle \cite{Bol94}]
\label{corA1.3}
Let  $A$ be a convex polygon  in $\R^2$, and $L$ be a lattice in $\R^2$. 
Then $A$ tiles at some level $k$ by translations with respect to $L$, if and only if 
$A$ is centrally symmetric, and for each pair of parallel edges $e$ and $e'$ of $A$
the following two conditions are satisfied:
\begin{enumerate-math}
	\item \label{corA1.3.1}
	for some $\lam \in L$, both $e + \lam$ and $e'$ lie on the same line;
	\item \label{corA1.3.2}
	if the vector $\vec{e}$ does not belongs to $L$,
	then the translation vector which carries $e$ onto $e'$, is in $L$.
\end{enumerate-math}
\end{corollary}

\begin{proof}
By \thmref{thmA1.1} it will be enough to show 
that the condition that $A$ is centrally symmetric and satisfies
 \ref{corA1.3.1} and \ref{corA1.3.2} for any 
pair of parallel edges $e$ and $e'$, is necessary and sufficient for the
 vanishing of all the Hadwiger  functionals $H_\Phi(A,L)$.
Below we will prove the necessity of the condition.
The sufficiency can be proved in a similar way.

We therefore assume that 
$H_\Phi(A,L)=0$ for all $0$- and $1$-flags $\Phi$.
Let $e$ be any edge of $A$, and let $\Phi$ be a $1$-flag determined by the line $\ell$
 containing $e$. Since $A$ is convex, there can be at most one other edge $e'$ which is 
parallel to $e$. The condition $H_\Phi(A,L)=0$ guarantees that such an edge $e'$
 indeed exists, and that $e'$ is of the same length as $e$. Since this is true for every
 edge $e$ of $A$, we conclude that $A$ is centrally symmetric. Moreover, for any edge $e$,
 the line $\ell'$ containing the parallel edge $e'$ must be of the form $\ell' = \ell + \lam$
for some $\lam \in L$, which implies condition \ref{corA1.3.1}.

To obtain condition  \ref{corA1.3.2} as well, we again let $e$ be one of the edges of $A$.
 Let $p$ be one of the endpoints of $e$, and $p'$ be the endpoint of the parallel edge $e'$
 that is given by $p'=p+\tau$, where $\tau$ is the translation vector which carries $e$ onto $e'$
 (see Figure \ref{fig:polygon}). Let $\Phi$ be a $0$-flag determined by the point $p$ and 
the line $\ell$ which contains the edge $e$. Then for the condition $H_\Phi(A,L)=0$ to hold,
the contribution of the point $p$ to the sum \eqref{eq:hadwiger} must be cancelled by
another term in the sum, which necessarily corresponds to a vertex of $A$ which lies in the 
set $L+p$, and moreover this vertex can be either the other endpoint of $e$ or the vertex $p'$. 
This shows that if the vector $\vec{e}$ is not in $L$, then $\tau$ must be in $L$, which
establishes condition  \ref{corA1.3.2} and thus concludes the proof.
\end{proof}

% -------------------------------------

\begin{figure}[htb]
\centering
\begin{tikzpicture}[scale=0.85, p2/.style={line width=0.275mm, black}, p3/.style={line width=0.15mm, black!50!white}]

% polygon
\draw[p2] (-1, -2) -- (2, -2);
\draw[p3] (2,-2) -- (3.5,-1) -- (3.5,0) -- (2.8, 1) -- (1, 2);
\draw[p2] (1, 2) -- (-2,2);
\draw[p3] (-2,2) -- (-3.5, 1) -- (-3.5, 0) -- (-2.8, -1) -- (-1,-2);

% vector tau
\draw[-stealth, p2] (-1,-2) -- (-1.975,1.9);

% text
\draw (-0.5, 2) node[anchor=south]{$e'$};
\draw (0.5, -2) node[anchor=north]{$e$};
\draw (-1.5, 0) node[anchor=west]{$\tau$};

% vertices
\fill (-1,-2) circle (0.07);
\draw (-1,-2) node[anchor=north] {\small $p$};
\fill (-2,2) circle (0.07);
\draw (-2,2) node[anchor=south east] {\small $p'$};

\end{tikzpicture}
\caption{The convex polygon $A$ in the proof of \corref{corA1.3}.}
\label{fig:polygon}
\end{figure}

% -------------------------------------

\subsection{}
Actually, the convexity assumption was not essential in the proof of \corref{corA1.3}. The main 
role of this assumption was to ensure that for each edge $e$ of $A$ there can be at most one 
other edge $e'$ which is parallel to $e$. Hence, basically the same proof is valid in this 
more general situation. In this case, \thmref{thmA1.1} recovers 
the following generalization of \corref{corA1.3}, which is due to Kolountzakis \cite[Theorem 2]{Kol00}:

\begin{corollary}[Kolountzakis \cite{Kol00}]
\label{corA1.10}
Let  $A$ be a polygon in $\R^2$ with the property that for each edge $e$ of $A$, there 
is at most one other edge $e'$  which is parallel to $e$. Then $A$ tiles at some level $k$ 
by translations with respect to a lattice $L$ in $\R^2$,  if and only if  for each edge 
$e$ of $A$ the following three conditions are satisfied:
$(i)$ there exists an edge $e'$ which is parallel to $e$, and $e'$ has the same length as $e$;
$(ii)$ for some $\lam \in L$, both $e + \lam$ and $e'$ lie on the same line; and
$(iii)$ if the vector $\vec{e}$ does not belongs to $L$,
then the translation vector which carries $e$ onto $e'$, is in $L$.
\end{corollary}

\subsection{}
\label{sec:symm}
For $k$-tiling by translates of a convex polytope in higher dimensions,
the following result was obtained in \cite[Theorem 1]{GRS12}:
\emph{If $A \subset \R^d$ is a convex polytope which $k$-tiles by
translations, then $A$ is centrally symmetric and has centrally symmetric facets.}

(Recall that a \emph{facet} of a polytope $A \subset \R^d$ is a $(d-1)$-dimensional face of $A$.)

Notice that this result is not specific for $k$-tilings by \emph{lattice} translates; the
conclusion is true regardless of whether the translation set is a lattice or not.

The proof of the result consists of two parts. In the first part, it is shown that the
 $k$-tiling assumption implies that the (classical) Hadwiger functionals $H_\Phi(A, \R^d)$
 (that are invariant with respect to all translations, not only lattice translations) vanish 
 for all $r$-flags $\Phi$ $(0 \leq r \leq d-1)$. In the special case of $k$-tiling with respect to a lattice $L$,
 this can be deduced from \thmref{thmA1.1}, since the condition that
$H_\Phi(A, L) = 0$ for every $\Phi$, implies in particular
that also $H_\Phi(A, \R^d) = 0$ for every $\Phi$.
In the general case (that is, for $k$-tiling with respect to an arbitrary
translation set), the vanishing of the Hadwiger functionals is
 proved using a volume growth argument, see 
\cite{Mur75}, \cite[Theorem 3]{LM95a}, \cite[Lemma 3.3]{GRS12}.
In fact, this part of the proof does not require the polytope $A$
 to be convex; the same argument is valid for an arbitrary polytope $A \subset \R^d$.

In the second part of the proof, the vanishing of the Hadwiger functionals
$H_\Phi(A, \R^d)$ is shown to imply that the convex polytope $A$ must
 be centrally symmetric and have centrally symmetric facets,
see \cite[Section 3.3]{Mur77}, \cite[pp.\ 640--641]{GRS12}. This is proved 
based on a classical theorem due to Minkowski, which states that
a convex polytope $A \subset \R^d$ is centrally symmetric if and only if every facet 
of $A$ has a parallel facet of the same $(d-1)$-dimensional volume.

\subsection{}
Next, we consider convex polytopes in three-dimensional space. In this case, we
will use \thmref{thmA1.1} to obtain  a condition which is necessary and sufficient for 
$k$-tiling by lattice translates,  similar in spirit to Bolle's criterion  in two dimensions.

To state the result, we will use the terminology from \cite{GKRS13} where the
notion of a \emph{four-legged-frame} was defined. The definition of a  four-legged-frame
will now be given.

Assume that $A \subset \R^3$ is a convex polytope, which is centrally symmetric and has centrally 
symmetric facets (as we have seen above, this condition is necessary for $A$ to 
$k$-tile by translations). Let $F$ be one of the facets of $A$. Then, due to the central 
symmetry assumptions, $F$ has a parallel facet $F'$ which is a translate of $F$.
Let $\tau'$ denote the translation vector which carries $F$ onto $F'$. 
Further, let $e$ be an edge of $A$ which is contained in $F$. Then  the central symmetry 
of $F$ implies that there is another edge $e''$ of $F$, which is parallel to $e$ and
has the same length. Let $\tau''$ be the translation vector which carries $e$ onto $e''$.
Finally, consider the four parallel edges of $A$ given by
\[
e, \quad e' = e + \tau', \quad e'' = e + \tau'', \quad e''' = e + \tau' + \tau'',
\]
see the illustration in Figure \ref{fig:polyhedron}. Any system $(e, e', e'', e''')$ of four
edges of  $A$ which can be obtained in this way, will be called 
(following \cite{GKRS13}) a \emph{four-legged-frame} of $A$.

% -------------------------------------

\begin{figure}[htb]
\centering

\begin{tikzpicture}[scale=0.6, p2/.style={line width=0.475mm, black}, p3/.style={line width=0.2mm, black!50!white}]

% facets fill
\fill [fill=gray!35]  (2,3) -- (0,1) -- (4,-0.5) -- (8,0) -- (10,2) -- (6,3.5) -- (2,3);
\fill [fill=gray!35]  (1,-3) -- (-1,-5) -- (3,-6.5) -- (7,-6) -- (9,-4) -- (5,-2.5) -- (1,-3);

% side facets
\draw [p3] (8,0) -- (8.5, -3) -- (7,-6);
\draw [p3] (10, 2) -- (10.5, -1) --  (9, -4);
\draw [p3] (8.5, -3) -- (10.5, -1);

% facets frame
\draw [p3] (2,3) -- (0,1) -- (4,-0.5) -- (8,0) -- (10,2) -- (6,3.5) -- (2,3);
\draw [p3] (1,-3) -- (-1,-5) -- (3,-6.5) -- (7,-6) -- (9,-4) -- (5,-2.5) -- (1,-3);

% bold edges
\draw [p2] (2,3) -- (0,1);
\draw [p2] (8,0) -- (10,2);
\draw [p2] (1,-3) -- (-1,-5);
\draw [p2] (7,-6) -- (9,-4);

% shift vectors
\draw[-stealth, line width=0.275mm, densely dashed] (0,1) -- (7.9,0.005);
\draw[-stealth, line width=0.275mm, densely dashed] (0,1) -- (-0.985,-4.91);

% tag names
\draw (0.7, 2.25) node{\small $e$};
\draw (8.65, 1.25) node{\small $e''$};
\draw (-0.05, -4) node[anchor=south]{\small $e'$};
\draw (8.6, -5.6) node[anchor=south]{\small $e'''$};
\draw (3.5, 1) node[anchor=west]{\small $\tau''$};
\draw (-0.85, -1) node{\small $\tau'$};
\draw (5.5, 2.5) node{$F$};
\draw (4.5, -3.5) node{$F'$};

\end{tikzpicture}
\caption{A four-legged-frame $(e, e', e'', e''')$ of a convex polytope in
$\R^3$ that is centrally symmetric and has centrally symmetric  facets.}
\label{fig:polyhedron}
\end{figure}
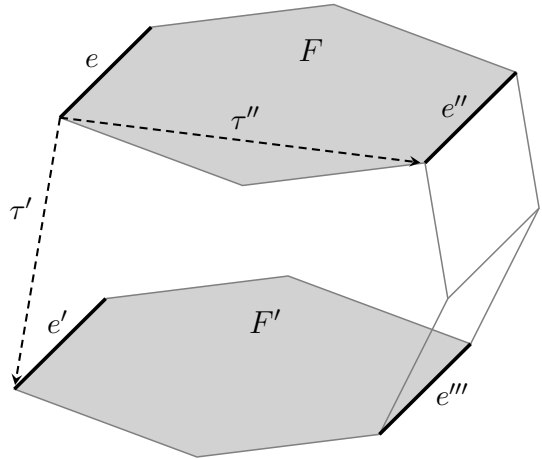

% -------------------------------------

We can now state our characterization of the three-dimensional 
convex polytopes which $k$-tile by  translations along a lattice:

\begin{thm}
\label{thmA1.4}
Let  $A$ be a convex polytope in $\R^3$, and $L$ be a lattice in $\R^3$. 
Then $A$ tiles at some level $k$ by translations with respect to $L$, if and only if 
$A$ is centrally symmetric, $A$ has centrally symmetric  facets, and for
each four-legged-frame $(e, e', e'', e''')$ of $A$ 
the following three conditions are satisfied:
\begin{enumerate-math}
	\item \label{corA1.4.1}
	for some vector $\lam' \in L$, both $F + \lam'$ and $F'$ lie on the same plane;

	\item \label{corA1.4.2}
	there is $\lam \in L$ such that  $e  + \lam$ lies on the same
	line with either $e'$ or $e''$;

	\item \label{corA1.4.3}
	if no one of the three vectors $\vec{e}$, $\tau'$, $\tau''$ belongs to $L$, then all
	the four vectors of the form $\vec{e} \pm \tau' \pm \tau''$ (corresponding
	to all possible choices of signs) are in $L$.
\end{enumerate-math}
The facets $F, F'$ and the vectors $\tau', \tau''$  which appear
in conditions \ref{corA1.4.1} and \ref{corA1.4.3}
correspond to the four-legged-frame $(e, e', e'', e''')$
in the way described above.
\end{thm}

\begin{proof}
By \thmref{thmA1.1} it will be enough to show 
that the above conditions are necessary and sufficient for the
 vanishing of all the Hadwiger  functionals $H_\Phi(A,L)$.
Below we will prove the necessity of these conditions.
The sufficiency can be proved in a similar way.

We therefore assume that $H_\Phi(A,L)=0$ for all $0$-, $1$- and $2$-flags $\Phi$.
Then we know that $A$ must be centrally symmetric and have centrally symmetric 
facets (see \secref{sec:symm}). 
Let $(e, e', e'', e''')$ be a  four-legged-frame of $A$, and let the facets $F, F'$ 
and the vectors $\tau', \tau''$  correspond to $(e, e', e'', e''')$ as described above.

Let $\Phi$ be a $2$-flag determined by the plane $P$ that contains the facet $F$.
 From the condition $H_\Phi(A,L)=0$ we infer that the parallel plane $P'$ which contains 
the  facet $F'$ must be of the form $P' = P+\lam'$ for some $\lam' \in L$,
which yields \ref{corA1.4.1}.

Next, we let $\Phi$ be a $1$-flag determined by the line $\ell$  that contains the
edge $e$, and by the plane $P$ containing  the facet $F$. For the condition $H_\Phi(A,L)=0$ to hold,
the contribution of the edge $e$ to the sum \eqref{eq:hadwiger} must be cancelled by
another term in the sum, which necessarily corresponds to another edge of $A$ that lies in some
line of the form $\ell + \lam$, $\lam \in L$, and this edge can be either $e'$ or $e''$.
Thus we obtain \ref{corA1.4.2}.

Finally, let $p$ and $q$ denote the two endpoints of the edge $e$. 
We divide the eight endpoints of the four edges $e, e', e'', e'''$ into two subsets
with four elements each, 
\begin{equation}
\label{eq:A3.1}
p, \quad q + \tau', \quad q + \tau'', \quad p + \tau' + \tau'',
\end{equation}
and
\begin{equation}
\label{eq:A3.2}
q, \quad p + \tau', \quad p + \tau'', \quad q + \tau' + \tau''.
\end{equation}
Let $\Phi$ be a $0$-flag that is determined by some point $x \in \R^3$, by a line $\ell$
through $x$  that is parallel to $e$, and by a plane $P$ containing $\ell$  which is
parallel to $F$. Then the value of $H_\Phi(A,L)$ is equal to the difference between the 
number of vertices from the subset \eqref{eq:A3.1} that belong to $L + x$,
  and the number of vertices from subset \eqref{eq:A3.2} that are in $L + x$.
One can verify that for the condition $H_\Phi(A,L)=0$ to hold regardless
of the choice of the point $x$, it is necessary and sufficient that at least one of
the three vectors $\vec{e}$, $\tau'$, $\tau''$ belong to $L$, or that all
the four vectors of the form $\vec{e} \pm \tau' \pm \tau''$ be in $L$.
This implies \ref{corA1.4.3}, and concludes the proof.
\end{proof}

\subsection{}
As our last example, we will use \thmref{thmA1.1} to recover \cite[Theorem 2]{GRS12},
which gives a useful sufficient condition for a polytope $A \subset \R^d$
to $k$-tile by lattice translates. Remark that 
 $A$ is not assumed to be convex, or connected, in this result.

\begin{corollary}[Gravin, Robins, Shiryaev \cite{GRS12}]
\label{corA2.1}
Let $A$ be a polytope in  $\R^d$, and $L$ be a lattice in $\R^d$.
 Suppose that $A$ is centrally symmetric, 
$A$ has centrally symmetric facets, and that all the vertices of $A$ lie in $L$. 
Then the translates of $A$ with respect to $L$ form a $k$-tiling for some $k$.
\end{corollary}

\begin{proof}
Let $t$ denote the center of symmetry of $A$, that is, $t$ is the unique 
point such that $A=-A+2t$. If $F$ is one of the facets of $A$,
 then by reflecting $F$ through the point $t$ we obtain another
facet $F' :=-F+2t$ of  $A$.  Let $t'$ denote the center of 
symmetry of the facet $F$, so that we have $F=-F+2t'$. It follows that 
$F'=F+2(t-t')$. Since all the vertices of $A$  lie in $L$, the translation 
vector $\lam := 2(t-t')$ which carries $F$ onto $F'$,
 must belong to $L$. Moreover, as $A$ and $F$ cannot have the same 
center of symmetry, we  have $t \neq t'$, and therefore $\lam \neq 0$.
Hence $F$ and $F'$ are two \emph{distinct} facets of $A$.

We thus conclude that the collection of all the facets of $A$ can be 
partitioned  into pairs of  parallel facets, such that if two facets 
$F$ and $F'$ constitute one of these pairs then $F'$ can be obtained 
from $F$ by reflection through the center of symmetry of $A$, 
and in addition there is a nonzero vector $\lam \in L$ 
such that $F' = F+\lam$. We will show that this  
implies that $H_{\Phi}(A,L) = 0$ for every 
$r$-flag $\Phi$ $(0 \leq r \leq d-1)$. Due to \thmref{thmA1.1},
this suffices in order to establish that $A$ tiles at some level $k$
by translations along $L$.

Let $0 \leq r \leq d-1$, and let 
\begin{equation}
\label{eq:A5.1}
F_r \subset F_{r+1} \subset \cdots  \subset F_{d-1}  \subset F_d = A
\end{equation}
be a sequence of faces of $A$, such that $F_j$ has dimension $j$ $(r \leq j \leq d-1)$. 
The facet $F_{d-1}$ from this sequence belongs to one
of the pairs of facets described above, so it is in the same pair with another
facet $F'_{d-1}$. Let $\lam \in L$ be the nonzero translation vector
such that $F'_{d-1} = F_{d-1} + \lam$. Then we can define another
sequence of faces
\begin{equation}
\label{eq:A5.2}
F'_r \subset F'_{r+1} \subset \cdots \subset F'_{d-1}  \subset F'_d = A
\end{equation}
given by $F'_j := F_j + \lam$ $(r \leq j \leq d-1)$. 

Suppose that  $\Psi$ is an $r$-flag determined by a sequence 
\begin{equation}
\label{eq:A5.3}
V_r \subset V_{r+1} \subset \cdots  \subset V_{d-1} \subset V_d = \R^d,
\end{equation}
such that $V_j$ is the $j$-dimensional affine subspace which contains
the face $F_j$ $(r \leq j \leq d-1)$, and let $\varepsilon_j$ be the $\pm 1$ 
coefficients associated to the sequence \eqref{eq:A5.1} with
respect to $\Psi$.  Let $\Psi'$ be the $r$-flag  obtained from $\Psi$ 
by translating all the affine subspaces in \eqref{eq:A5.3},
 as well as the positive and negative half-spaces,  by the vector $\lam$
(remark that it may happen that $\Psi'$ coincides with $\Psi$), and
 let $\varepsilon'_j$ be the  $\pm 1$  coefficients associated to the sequence
 \eqref{eq:A5.2} with respect to $\Psi'$. Then we have
\[
\eps'_j = \eps_j \quad (r \leq j < d-1) \quad \text{and} \quad \eps'_{d-1} = - \eps_{d-1},
\]
where the last equality is due to the fact that the facet
$F'_{d-1}$ can be obtained from $F_{d-1}$ by reflection through 
the center of symmetry of $A$. This implies that
\[
\varepsilon_{r} \varepsilon_{r+1} \cdots  \varepsilon_{d-1} \Vol_r (F_r) + 
\varepsilon'_{r} \varepsilon'_{r+1} \cdots \varepsilon'_{d-1} \Vol_r (F'_r)  = 0.
\]

We conclude that for any given $r$-flag $\Phi$ $(0 \leq r \leq d-1)$, the
collection of all the $r$-dimensional faces of $A$ which contribute to
the sum \eqref{eq:hadwiger} can be partitioned into pairs, such that
the joint contribution of two $r$-dimensional faces in the same pair is equal to zero.
Hence the entire sum in \eqref{eq:hadwiger} must be zero, and so we
obtain that $H_{\Phi}(A,L) = 0$ as we had to show.
\end{proof}

\begin{remark}
It can be seen from the proof above that the assertion
in \corref{corA2.1} remains true if we relax the assumption  
that the all the vertices of $A$ lie in $L$, and instead only
require that each translation vector mapping 
a facet $F$ onto the parallel facet  $F'$ obtained 
by reflecting $F$ through the center of symmetry of $A$, 
belongs to $L$.
\end{remark}

% =========================================================

\section{Preliminaries}
\label{sec:prelim}

The rest of the paper is devoted to the proofs of our results.
The present section contains some preliminary material that will be used
in the proofs.

\subsection{Notation}
If $A$ is a subset of $\R^d$ and $\lam$ is a vector in $\R^d$,
then we denote by $A + \lam$ the translation of $A$ by the vector
$\lam$. If $A, B$ are two subsets of $\R^d$, then
$A+B$ and $A-B$ denote their set of sums and set of 
differences respectively.

We use $\dotprod{\cdot}{\cdot}$ to denote the standard scalar product in $\R^d$.

By a (full-rank) \emph{lattice} $L$ in $\R^d$ we  refer to a discrete subgroup 
generated by $d$ linearly independent vectors. We use $\det(L)$ to denote
the volume of some (or equivalently, of any) fundamental parallelepiped
of $L$. The \emph{dual lattice} $L^*$ is the set of all vectors $\lambda^* \in \R^d$ 
such that $\dotprod{\lambda}{\lambda^*}  \in \Z$, $\lambda \in L$.

For each $\xi \in \R^d$ we use the notation
$e_\xi (x) := e^{2\pi i\langle \xi,x\rangle}$,  $x \in \R^d$.

The Fourier transform of a function $f \in L^1(\R^d)$ is normalized to be
\[
\ft f (\xi)=\int_{\R^d} f (x) \, \overline{e_\xi (x)}  dx,
\]
and similarly, the Fourier transform of a finite, complex measure $\mu$ on $\R^d$ is
\[
\ft{\mu} (\xi)=\int_{\R^d} \overline{e_\xi (x)}  \,  d\mu(x).
\]

\subsection{The polytope group}
By a \emph{polytope} in $\R^d$ we mean a set 
which can be represented as the union of a finite number of 
simplices with disjoint interiors, where a \emph{simplex} is the convex 
hull of $d + 1$ points which do not all lie in some hyperplane.

It will be convenient to generalize the notion of a polytope by
considering elements of the \emph{polytope group} $\P^d$.
We define the polytope group $\P^d$ to be the abelian group
generated by the elements $[A]$ where $A$ goes through
all polytopes in $\R^d$,  subject to the relations
$[A] + [B] = [A \cup B]$ whenever $A$ and $B$ are two
polytopes with disjoint interiors.
Any element $P$ of the polytope group $\P^d$ can be
represented as a finite sum
\begin{equation}
\label{eq:P1.2}
P = \sum_{j} m_j [A_j],
\end{equation}
where $m_j$ are distinct nonzero integers, and $A_j$
are polytopes with pairwise disjoint interiors. This representation is
unique up to the order of the terms in the sum.

We think of the set of polytopes in  $\R^d$ as a subset 
of the polytope group $\P^d$, by identifying a polytope $A$
with the element $[A]$ of $\P^d$.

A function $\varphi$, defined on the set of all polytopes in 
$\mathbb R^d$, is said to be \emph{additive} if we have
$\varphi(A \cup B) = \varphi(A) + \varphi(B)$
whenever $A$ and $B$  are two polytopes with disjoint interiors.
Any such a function $\varphi$ can be extended in a unique way
to a function (also denoted by $\varphi$) on the  polytope group $\P^d$, satisfying
$\varphi(P) = \varphi(P') + \varphi(P'')$  whenever $P = P' + P''$.
The extension is given by 
$\varphi(P) = \sum m_j \varphi(A_j)$
for an element $P$ of the form \eqref{eq:P1.2}.

For example, the \emph{$d$-dimensional volume} is an additive
function on the set of polytopes in $\R^d$. Hence we can
extend the notion of the volume by additivity to all the elements of the polytope 
group $\P^d$. We will use
$\vol(P)$ to denote the volume of an element $P \in \P^d$.
Notice that $\vol(P)$ is a real number 
which may be positive, negative or zero.

We can also define, in a similar way, the \emph{indicator function} $\chi_P$
of an element $P$ of the polytope  group $\P^d$. Indeed, the
 mapping which takes a polytope $A$ to its indicator 
function $\chi_A$ is an additive mapping from the
set of polytopes into $L^1(\R^d)$. Hence this mapping extends by additivity
to the polytope group $\P^d$, and thus for each $P \in \P^d$ we 
define its indicator function $\chi_P$ as an element  of the space $L^1(\R^d)$.

For any element $P$ of the polytope group $\P^d$ we have
$\int \chi_P(x) dx = \vol(P)$.
Indeed, this equality is obviously true if $P$ is a polytope, and 
by additivity it therefore holds for all the elements of $\P^d$.

\subsection{Tiling}
We say that an element  $P$ of the polytope group $\P^d$ \emph{tiles} at level $k$
by translations with respect to a lattice $L$, if we have
\begin{equation}
\label{eq:P1.3}
\sum_{\lam \in L} \chi_P(x - \lam) = k \quad \text{a.e.}
\end{equation}
Since the indicator function $\chi_P$ is integer-valued a.e., 
the tiling level $k$ is necessarily an integer, which may be positive, negative or zero.

\begin{prop}
\label{propD1.1}
Let $P$ be an element of the polytope group $\P^d$, and let $L$ be a lattice in $\R^d$.
If $P$ tiles at level $k$ by translations with respect to $L$, then
$\vol(P) = k \det(L)$.
\end{prop}

\begin{proof}
Let $D$ be a fundamental parallelepiped of $L$. By \eqref{eq:P1.3} we have
\begin{equation}
\label{eq:lemD1.1.1}
\int_D \sum_{\lam \in L} \chi_P(x - \lam) dx = k \vol(D).
\end{equation}
On the other hand, the left hand side of \eqref{eq:lemD1.1.1} is equal to
\[
 \sum_{\lam \in L} \int_D \chi_P(x - \lam) dx 
= \sum_{\lam \in L} \int_{D - \lam} \chi_P(x) dx 
= \int_{\R^d} \chi_P(x) dx,
\]
where the last equality holds since $D$ tiles (at level $1$) by translations
with respect to $L$. Since $\vol(D) = \det(L)$ and $\int \chi_P(x) dx = \vol(P)$,
this proves the claim.
\end{proof}

\begin{prop}
\label{propD1.4}
Let $P$ be an element of the polytope group $\P^d$, and $L$ be a lattice in $\R^d$.
Then $P$ tiles at some level $k$ by translations with respect to $L$, 
if and only if the function $\ft{\chi}_P$ $($the Fourier transform 
of $\chi_P$$)$ vanishes on $L^* \setminus \{0\}$.
\end{prop}

\begin{proof}
By applying a linear transformation we may suppose that $L = \Z^d$. Let
\[
f(x) := \sum_{m \in \Z^d} \chi_P(x - m),
\]
then $f$ is a $\Z^d$-periodic function whose Fourier series is given by
\[
\sum_{m \in \Z^d}  \ft{\chi}_P(m) e^{2 \pi i \dotprod{m}{x}}
\]
(this is one of the variants of Poisson's summation formula, see e.g.\ \cite[Chapter VII, Theorem 2.4]{SW71}).
Hence $f$ coincides a.e.\ with a constant function, if and only if $\ft{\chi}_P(m)=0$
for every $m \in \Z^d \setminus \{0\}$. This yields the claim.
\end{proof}

\subsection{Equidecomposability}
Let $A$ and $B$ be two polytopes in $\R^d$, and let $L$ be a lattice in $\R^d$. We say that
$A$ and $B$ are \emph{equidecomposable with respect to translations by vectors from $L$},
if there exist finite decompositions of $A$ and $B$ of the form
\[
A = \bigcup_{j=1}^{N} A_j, \quad  B = \bigcup_{j=1}^{N} B_j
\]
where $A_1, \dots, A_N$ are polytopes with pairwise disjoint interiors,
$B_1, \dots, B_N$ are also polytopes with pairwise disjoint interiors,
and $B_j$ can be obtained from $A_j$ by translation along some vector
$\lam_j$ belonging to $L$.

The following proposition connects the two notions of
equidecomposability and tiling by lattice translations.

\begin{prop}
\label{propD1.2}
Let $A$ and $B$ be two polytopes in $\R^d$, and $L$ be a lattice in $\R^d$.
Then the following two conditions are equivalent:
\begin{enumerate-math}
	\item \label{lemD1.2.1}
	$A$ and $B$ are equidecomposable with respect to translations by vectors from $L$;
	\item \label{lemD1.2.2}
	The element $[A] - [B]$ of the polytope group $\P^d$
	tiles at level zero by translations with respect to $L$.
\end{enumerate-math}
\end{prop}

\begin{proof}
We start by showing that \ref{lemD1.2.1} implies \ref{lemD1.2.2}. The condition
\ref{lemD1.2.1} means that we have decompositions
$[A]  = [A_1] + \dots + [A_n]$ and
$[B] = [B_1] +\dots+ [B_n]$, where $A_j, B_j$ are polytopes such
that $B_j$ is obtained from $A_j$ by  translation along some vector 
$\lam_j \in L$. Hence, for each $j$, the element $[A_j] - [B_j]$ is easily seen
to tile at level zero by translations with respect to $L$.
Since $[A]-[B] = \sum_{j=1}^{n} ([A_j] - [B_j])$, condition
\ref{lemD1.2.2} follows.

Next, we prove that \ref{lemD1.2.2} implies \ref{lemD1.2.1}. The proof
is done by induction on the number of elements of the set
\[
S(A,B) := \{ \lam \in L : \vol(A \cap (B - \lam)) > 0 \}.
\]
If the set $S(A,B)$ is empty, then it follows from \ref{lemD1.2.2} that 
both $A$ and $B$ must be empty sets, so in this case the condition
\ref{lemD1.2.1} holds trivially. Otherwise, the set $S(A,B)$ is nonempty,
hence  we may choose some vector $\lam \in S(A,B)$. Then the two sets
\[
A' := A \cap (B - \lam), \quad B' := (A + \lam) \cap B
\]
are polytopes satisfying $B' = A' + \lam$. Since $A' \subset A$ we
may decompose $A$ as the union of $A'$ and another (possibly empty)
polytope $A''$,  where $A'$, $A''$ have disjoint interiors. Similarly, $B$ 
is the union of $B'$ and some polytope $B''$  (which, again, may be empty)
such that the interiors of $B'$ and $B''$ are disjoint. It follows from \ref{lemD1.2.2}
that the element $[A''] - [B'']$ of the polytope group $\P^d$
tiles at level zero by translations with respect to $L$. Moreover, the set
$S(A'', B'')$ is a subset of $S(A, B) \setminus \{\lam\}$, and therefore
it has less elements than $S(A, B)$. So the inductive hypothesis implies that 
$A''$ and $B''$ are equidecomposable using translations by vectors from $L$.
Hence the same is true also for $A$ and $B$, and thus \ref{lemD1.2.1} is established.
\end{proof}

As a consequence of \propref{propD1.2} we  obtain \propref{propA2.2} stated above:

\begin{proof}[Proof of \propref{propA2.2}]
Let $A$ and $B$ be two polytopes in $\R^d$, let $L$ be a lattice in $\R^d$,
and assume that $B$ tiles at some level $k$ by translations with respect to $L$. 
Then the translates of $A$ along $L$ is also a $k$-tiling, if and only if the element
$[A]-[B]$ of the polytope group $\P^d$ tiles at level zero by translations along $L$.
By \propref{propD1.2}, the latter condition is equivalent to $A$ and $B$
  being equidecomposable with respect to translations by vectors from $L$.
\end{proof}

We also use \propref{propD1.2} to establish the following claim:

\begin{prop}
\label{propE1.2}
Let $P$ be an element of the polytope group $\P^d$, and let $L$ be a lattice in $\R^d$.
Then $P$ tiles at level zero by translations with respect to $L$, if and only if
$P$ can be represented as a finite sum $P = \sum ([A_j] - [A'_j])$ where
$A_j$, $A'_j$ are polytopes such that $A'_j$ is obtained from $A_j$ by
translation along a vector $\lam_j \in L$.
\end{prop}

\begin{proof}
The ``if'' part of the claim is obvious, so we will just prove the ``only if''
part. Let $\mathscr{C}^d(L)$ denote the subgroup of the polytope group $\P^d$
generated by all the elements of the form $[A]-[A']$, such that $A,A'$ are polytopes
and $A'$ is obtained from $A$ by translation along a vector from $L$.
It is enough to show that if $P$ tiles at level zero by translations 
with respect to $L$, then $P$ belongs to the subgroup $\mathscr{C}^d(L)$.

The polytope group $\P^d$ is generated by the elements $[A]$, where $A$ goes through 
all the polytopes. We can therefore represent $P$ in the form
\[
P = [A_1] + \dots + [A_n] - [B_1] - \dots - [B_m],
\]
where $A_j$, $B_k$ are polytopes. For each $j$ we let $A'_j$ be a polytope
obtained from $A_j$ by translation along some vector in $L$, where the translation
vectors are chosen such  that the  polytopes $A'_1, \dots, A'_n$ have pairwise
disjoint interiors. Hence $[A'_1] + \dots + [A'_n] = [A']$ for a certain polytope $A'$. 
In the same way, for each $k$ we let $B'_k$ be a polytope obtained from $B_k$ by 
translation along some vector in $L$, such  that $[B'_1] +\dots+ [B'_m] = [B']$
for some polytope $B'$. Let $P' :=  [A'] - [B']$, then
\[
P - P' = \sum_{j=1}^{n} ([A_j] - [A'_j]) - \sum_{k=1}^{m} ([B_k] - [B'_k]),
\]
which shows that $P - P'$ belongs to the subgroup  $\mathscr{C}^d(L)$. In particular,
this implies that $P - P'$ tiles at level zero by translations with respect to $L$. Hence
also $P' = P - (P - P')$ tiles at level 
zero by translations with respect to $L$. By \propref{propD1.2},  the polytopes
$A'$ and $B'$ are therefore equidecomposable with respect to translations by vectors from $L$.
This implies that $P'$ belongs to the subgroup  $\mathscr{C}^d(L)$. We obtain that
both $P'$ and $P-P'$ are in $\mathscr{C}^d(L)$, and therefore $P$ is in $\mathscr{C}^d(L)$
as well, as we had to show.
\end{proof}

\subsection{Flags}
If $r$ is an integer, $0 \leq r \leq d-1$, then
an \emph{$r$-flag} $\Phi$ in $\R^d$ 
is defined to be a sequence of affine subspaces 
\begin{equation}
\label{eq:C3.1.1}
V_r \subset V_{r+1} \subset \cdots  \subset V_{d-1} \subset V_d = \R^d
\end{equation}
such that $V_j$ has dimension $j$. Each subspace $V_j$ ($r \leq j \leq d-1$)
 in the sequence divides the next one $V_{j+1}$ into two half-spaces;
we assume that $\Phi$ is endowed with a choice of one of these
half-spaces being called positive, and the other being called negative.

It will be convenient to define also a \emph{$d$-flag} in $\R^d$ to be the sequence 
which consists of just one subspace $V_d = \R^d$. 

Let $A$ be a polytope in $\R^d$, and suppose that we have a sequence
\[
F_r \subset F_{r+1} \subset \cdots \subset F_{d-1} \subset F_d = A,
\]
where $F_j$ is a $j$-dimensional face of $A$ $(r \leq j \leq d-1)$. Such a sequence will be called 
an \emph{$r$-sequence} of faces of the polytope $A$, and will be denoted
by $\F_r$.

 We will say that the $r$-sequence $\F_r$ is contained in the
$r$-flag $\Phi$ if the face $F_j$ is contained in the subspace $V_j$ for each 
$r \leq j \leq d-1$.

Each $r$-flag $\Phi$ $(0 \leq r \leq d)$
determines an additive ``weight'' function $\omega_{\Phi}$,
which is defined by \eqref{eq:weight} on the set of polytopes, and is extended
by additivity to the polytope group $\P^d$. Notice that if $\Phi$ is a $d$-flag, then
$\omega_{\Phi}(P)=\vol(P)$ for every $P \in \P^d$.

The \emph{Hadwiger functional} $H_\Phi(\,\cdot\, , L)$ associated to an
$r$-flag $\Phi$ $(0 \leq r \leq d)$ and to a lattice $L$ in $\R^d$, is
the additive function on the polytope group $\P^d$ which is given
by \eqref{eq:hadwiger} on the set of polytopes. The Hadwiger functional
is invariant with respect to translations by vectors from $L$, that is,
$H_\Phi(P, L) = H_\Phi(Q, L)$ if $Q$ is obtained from $P$ by
translation along a vector $\lam \in L$.

\subsection{Flag measures}
Let $\Phi$ be an $r$-flag   in $\R^d$ $(0 \leq r \leq d)$,
determined by a sequence of affine subspaces  \eqref{eq:C3.1.1}.
For each polytope $A$ we define a signed measure $\nu_{A,\Phi}$ 
on $\R^d$ given by
\begin{equation}
\label{eq:C3.1.2}
\nu_{A,\Phi} = \sum_{\F_r} \varepsilon_{r} \varepsilon_{r+1} \cdots \varepsilon_{d-1} \, {\Vol_r}|_{F_r} ,
\end{equation}
where $\F_r$  goes through all $r$-sequences of faces
$F_r \subset F_{r+1} \subset \cdots \subset F_d$
of the polytope $A$ that are contained in $\Phi$,
 the $\varepsilon_j$ are the $\pm 1$ coefficients 
associated to the $r$-sequence $\F_r$ with
respect to $\Phi$ in the same way as in \eqref{eq:weight},
 and where ${\Vol_r}|_{F_r}$ denotes the $r$-dimensional volume measure
restricted to the face $F_r$.

In the case when $r=0$, by a $0$-dimensional face of $A$ we mean a vertex of $A$, and 
the measure ${\Vol_r}|_{F_r}$ is  understood to be the Dirac measure at the vertex $F_r$.

The mapping which takes a polytope $A$ to the measure
$\nu_{A,\Phi}$ can be seen to be an additive mapping, that is,
$\nu_{A \cup B,\Phi} = \nu_{A,\Phi} + \nu_{B,\Phi}$
whenever  $A$ and $B$  are two polytopes with disjoint interiors. 
Hence we can extend this mapping by additivity, and define
the signed measure $\nu_{P,\Phi}$ for any
element $P$  of the polytope group $\P^d$.

Remark that if $\Phi$ is a $d$-flag, then $\nu_{P,\Phi}$ is
the measure $\chi_P(x) dx$, that is, an
absolutely continuous measure with density $\chi_P$
with respect to the Lebesgue measure $dx$.

It follows from \eqref{eq:weight} and \eqref{eq:C3.1.2} that the
measure $\nu_{P,\Phi}$ satisfies
$\int d\nu_{P,\Phi} = \omega_{\Phi}(P)$,
for any $r$-flag $\Phi$   $(0 \leq r \leq d)$ and any 
element $P$  of the polytope group $\P^d$.

If $\Phi$ is an $r$-flag in $\R^d$ $(0 \leq r \leq d)$,
$L$ is a lattice in $\R^d$, and $P$ is 
an element  of the polytope group $\P^d$,
 then we define the signed measure
\begin{equation}
\label{eq:C3.1.4}
\mu_{P,\Phi,L} = \sum_{\Psi} \nu_{P,\Psi}, 
\end{equation}
where $\Psi$ runs through all $r$-flags for which there exists $\lam \in L$ such that $\Psi$
can be obtained from $\Phi$ by translating all the 
affine subspaces in \eqref{eq:C3.1.1}, as well as the positive and negative half-spaces, 
 by the vector $\lam$. As a function of $P$, the measure $\mu_{P,\Phi,L}$ is additive 
and invariant under translations by vectors from $L$. It  satisfies
\begin{equation}
\label{eq:C3.1.5}
\int d \mu_{P,\Phi,L}  = H_{\Phi}(P,L).
\end{equation}

% =========================================================

\section{Proof of main result}

\subsection{}
In this section we prove the following theorem:

\begin{thm}
\label{thmC3.1}
Let $P$ be an element of the polytope group $\P^d$, and let $L$ be a lattice in $\R^d$.
Then $P$ tiles at some level $k$ by translations with respect to $L$, if and only if
$H_{\Phi}(P,L) = 0$ for every $r$-flag $\Phi$ $(0 \leq r \leq d-1)$.
\end{thm}

This is, in fact, the main result of this paper. It generalizes the
statement of Theorem \ref{thmA1.1}, from the set of  polytopes to all the
elements of the polytope group $\P^d$. In particular, 
\thmref{thmA1.1} follows from \thmref{thmC3.1} in the
special case when $P$ is a polytope.

We also deduce \thmref{thmA2.1} from \thmref{thmC3.1}. This is done as follows:

\begin{proof}[Proof of \thmref{thmA2.1}]
Let $A$ and $B$ be two polytopes in $\R^d$, and let $L$ be a lattice in $\R^d$. 
If $A$ and $B$ are equidecomposable with respect to translations by vectors 
from $L$, then it is obvious that $A$ and $B$ have the same
volume, and
\begin{equation}
\label{eq:C8.1.1}
H_{\Phi}(A,L) = H_{\Phi}(B,L) 
\end{equation}
for every $r$-flag $\Phi$ $(0 \leq r \leq d-1)$. This is so because the volume
function, as well as all the Hadwiger functionals  $H_\Phi(\,\cdot\, , L)$,
are additive functions on the set of polytopes  that are invariant
with respect to translations by vectors from $L$. 

Conversely,  suppose that $A$ and $B$ have the same volume, and that
\eqref{eq:C8.1.1} holds for every $r$-flag $\Phi$ $(0 \leq r \leq d-1)$. 
Then by additivity, the element $P := [A] - [B]$ of the polytope group 
$\P^d$ has volume zero, and it satisfies 
$H_{\Phi}(P,L) = 0$ for every $r$-flag $\Phi$ $(0 \leq r \leq d-1)$.
 \thmref{thmC3.1} therefore yields that $P$ tiles at some level $k$ by 
translations with respect to $L$. By \propref{propD1.1} we have $\vol(P) = k \det(L)$, 
hence the tiling level $k$ must be zero. We can therefore invoke \propref{propD1.2},
which yields that $A$ and $B$ are equidecomposable with respect to translations by vectors from $L$.
\end{proof}

\subsection{}
The remainder of this section is devoted to the proof
of \thmref{thmC3.1}. The next lemma is the key one in the proof:

\begin{lem}
\label{lemC1.1}
Let $P$ be an element  of the polytope group $\P^d$,
let $L$ be a lattice in $\R^d$,
and let $r$ be an integer, $0 \leq r \leq d-1$. Assume that
$H_{\Psi}(P,L) = 0$ for every $s$-flag $\Psi$
 $(0 \leq s \leq r)$. Then for every $r$-flag $\Phi$ we have
\begin{equation}
\label{eq:C1.1}
\ft{\mu}_{P,\Phi,L}(\xi) = 0
\end{equation}
for all $\xi \in L^*$.
\end{lem}

\begin{proof}
Let $\Phi$ be an $r$-flag determined by a sequence of affine subspaces  
\begin{equation}
\label{eq:D1.1}
V_r \subset V_{r+1} \subset \cdots  \subset V_{d-1} \subset V_d = \R^d,
\end{equation}
such that $V_j$ has dimension $j$. 
We must show that \eqref{eq:C1.1} holds for  all $\xi \in L^*$.
We prove this by induction on $r$.

By applying a translation we may assume that all the
affine subspaces in \eqref{eq:D1.1} are in fact linear subspaces.
We will first show that \eqref{eq:C1.1} is true for every $\xi \in L^* \cap V_r^\perp$, where 
$V_r^\perp$ denotes the orthogonal complement of the linear subspace $V_r$.
Indeed, for such $\xi$ we have
$e_\xi(x) = 1$ for all $x$ belonging to the support of 
 the measure $\mu_{P,\Phi,L}$, since the support of $\mu_{P,\Phi,L}$ is contained in
the set $V_r + L$ and we have $\dotprod{x}{\xi} \in \Z$ whenever $x$ is a point in $V_r + L$.
This implies that
\begin{equation}
\ft{\mu}_{P,\Phi,L}(\xi) = \int \overline{e_\xi} \, d \mu_{P,\Phi,L} = \int d \mu_{P,\Phi,L}  = H_{\Phi}(P,L).
\end{equation}
Hence  \eqref{eq:C1.1} follows using the assumption
 that $H_{\Phi}(P,L) = 0$.

Notice that if $r=0$, then the subspace $V_r^\perp$ coincides
with the whole space $\R^d$. So in this case,
the full assertion of \lemref{lemC1.1} is already established.
Hence, in what follows we will assume  that $r \geq 1$. In this case,
it still remains for us to show that \eqref{eq:C1.1} is true also for every 
$\xi \in L^* \setminus V_r^\perp$.

The polytope group $\P^d$ is generated by the elements $[A]$, where $A$ goes through 
all the polytopes in $\R^d$. We can therefore represent $P$ as a finite sum
$P=\sum m_k [A_k]$, where $m_k$ are integers and $A_k$ are polytopes.
By additivity we have 
\begin{equation}
\label{sum-muj}
 \mu_{P,\Phi,L} = \sum_k m_k \, \mu_{A_k,\Phi,L}.
\end{equation}
The Fourier transform of the measure $\mu_{A_k,\Phi,L}$ at the point $\xi$ is given by
\begin{equation}
\label{ft-muj}
\ft{\mu}_{A_k,\Phi,L}(\xi) = \int \overline{e_\xi} \, d\mu_{A_k,\Phi,L} =
\sum_{\Psi}\sum_{\F_r}  \eps_r \eps_{r+1} \cdots\eps_{d-1}\int_{F_r} \overline{e_\xi},
\end{equation}
where $\Psi$ runs through all $r$-flags that can be obtained from $\Phi$
by translation along a vector from $L$, $\F_r$ goes through all $r$-sequences of faces
$F_r \subset F_{r+1} \subset \cdots \subset F_d$
of the polytope $A_k$ that are contained in $\Psi$, 
the $\varepsilon_j$'s are the $\pm 1$ coefficients associated to the $r$-sequence
$\F_r$ with respect to $\Psi$, and the integral on the right hand side
is taken with respect  to the $r$-dimensional volume measure on the face $F_r$.

Let $\partial F_r$ denote the relative boundary of the face $F_r$, and for each
$x \in \partial F_r$ let $n(x)$ be a vector in the linear subspace 
$V_r$ which is outward unit normal  to $F_r$ at the point $x$.
Then for every $v \in V_r$ we have
\begin{equation}
\label{div-thm}
-2\pi i \dotprod{\xi}{v} 
\int_{F_r} \overline{e_\xi} = \int_{\partial F_r} \dotprod{n}{v} \overline{e_\xi},
\end{equation}
which follows by applying the divergence theorem to the 
function $f(x) = v \, \overline{e_\xi (x)}$ over the face $F_r$.
The relative boundary $\partial F_r$ consists of a finite number
of $(r-1)$-dimensional faces $F_{r-1}$ of $F_r$.
Hence, using \eqref{ft-muj} and \eqref{div-thm}, we get
\begin{align}
\label{sum-in-r-1-flags-a}
&-2\pi i \dotprod{\xi}{v}  \ft{\mu}_{A_k,\Phi,L}(\xi) 
= \sum_{\Psi}\sum_{\F_r}  \eps_r \eps_{r+1}  \cdots\eps_{d-1} \int_{\partial F_r}
 \dotprod{n}{v} \overline{e_\xi} \\
\label{sum-in-r-1-flags-b}
&\qquad = \sum_{\Psi} \sum_{\F_r} \eps_r \eps_{r+1}  \cdots\eps_{d-1}
 \sum_{F_{r-1}} \dotprod{n}{v} \int_{ F_{r-1}} \, \overline{e_\xi},
\end{align}
where $F_{r-1}$ goes through the $(r-1)$-dimensional faces of 
the $r$-dimensional face $F_r$ from  the sequence $\F_r$,
 and $n$ is the outward unit normal to $F_r$ on $F_{r-1}$.

Let $\E$ be the collection of all the $(r-1)$-sequences  of faces 
$F_{r-1} \subset F_r \subset \cdots \subset F_d$ that belong to one
of the polytopes $A_k$ (where $k$ may be different from one
sequence to another, so that $k$ is not fixed), and such that the $r$-subsequence
$F_r \subset F_{r+1} \subset \cdots \subset F_d$ is contained in some
$r$-flag $\Psi$ that can be obtained from $\Phi$
by translation along a vector from $L$.
We define an equivalence relation on $\E$ by saying that two elements 
$\F_{r-1}$ and $\F'_{r-1}$ from $\E$ are equivalent
if the affine hull of the $(r-1)$-dimensional face $F_{r-1}$ from the sequence $\F_{r-1}$ 
can be translated by a vector in $L$ onto the affine hull of the $(r-1)$-dimensional face 
$F'_{r-1}$ from $\F'_{r-1}$. Then $\E$ can be partitioned  into a finite number of
equivalence classes $\E^1, \E^2, \dots, \E^N$ induced
by this  equivalence relation.

To each equivalence class $\E^l$ $(1 \leq l \leq N)$ we associate an $(r-1)$-flag $\Phi^l$, 
defined in the following way. The flag $\Phi^l$ is 
determined by a sequence of affine subspaces  
\[
V_{r-1}^l \subset V_r \subset V_{r+1} \subset \cdots  \subset V_d = \R^d,
\]
where $V_r, V_{r+1}, \dots, V_d$ are the same linear subspaces from \eqref{eq:D1.1} that 
determine the $r$-flag $\Phi$, while $V^l_{r-1}$ is a new affine subspace of
dimension $r-1$. The subspace $V^l_{r-1}$  is chosen in such a way that the set
$V^l_{r-1} + L$ contains all the $(r-1)$-dimensional faces $F_{r-1}$ belonging to
sequences  $\F_{r-1}$ from the equivalence class $\E^l$. It is straightforward to verify
that such a choice of $V^l_{r-1}$  is indeed possible, thanks to the definition of the 
collection $\E$ and the way in which the equivalence relation on $\E$ was defined.
 We endow the $(r-1)$-flag $\Phi^l$ with a choice of positive and negative half-spaces, 
by saying that the positive and negative half-spaces of $V_{j+1}$ determined
by the subspace $V_j$ coincide with those from the $r$-flag $\Phi$
for all $r \leq j \leq d-1$; while the positive and negative half-spaces 
of $V_r$ that are determined by the new subspace $V^l_{r-1}$
are selected in an arbitrary way.

For each $1 \leq l \leq N$, let $\sigma^l$ denote the (unique) unit vector in the linear subspace 
$V_r$ which is normal to $V^l_{r-1}$ and is pointing towards the negative half-space 
of $V_r$ determined  by $V^l_{r-1}$. We then observe that for any given $k$, the sum in
\eqref{sum-in-r-1-flags-b} is equal to
\begin{equation}
\label{sum-in-r-1-flags-c}
\sum_{l=1}^{N} \dotprod{\sigma^l}{v}
\sum_{\Theta} \sum_{\F_{r-1}} 
 \eps_{r-1} \eps_r \eps_{r+1}  \cdots\eps_{d-1}\int_{ F_{r-1}} \, \overline{e_\xi},
\end{equation}
where $\Theta$ runs through all $(r-1)$-flags that can be obtained from $\Phi^l$
by translation along a vector from $L$, $\F_{r-1}$ goes through all $(r-1)$-sequences of faces
$F_{r-1} \subset F_r \subset \cdots \subset F_d$
of the polytope $A_k$ that are contained in $\Theta$, and
the $\varepsilon_j$'s are the $\pm 1$ coefficients associated to the $(r-1)$-sequence
$\F_{r-1}$ with respect to $\Theta$. Hence
\eqref{sum-in-r-1-flags-a}, \eqref{sum-in-r-1-flags-b} and \eqref{sum-in-r-1-flags-c}
give that
\begin{equation}
\label{eq:D1.19}
-2\pi i \dotprod{\xi}{v}  \ft{\mu}_{A_k,\Phi,L}(\xi) =
\sum_{l=1}^{N} \dotprod{\sigma^l}{v} \, \ft{\mu}_{A_k,\Phi^l,L}(\xi),
\end{equation}
for each $k$. Finally, by multiplying both sides of \eqref{eq:D1.19} by $m_k$ and then
taking the sum with respect to $k$, we conclude that
\begin{equation}
\label{eq:D1.20}
-2\pi i \dotprod{\xi}{v}  \ft{\mu}_{P,\Phi,L}(\xi) 
= \sum_{l=1}^{N} \dotprod{\sigma^l}{v}  \ft{\mu}_{P,\Phi^l,L}(\xi)
\end{equation}
for every $\xi \in \R^d$ and every $v \in V_r$.

Now recall that it remained to show that \eqref{eq:C1.1} is true 
for every $\xi \in L^* \setminus V_r^\perp$.
Since $\xi \in L^*$,  it follows from the inductive hypothesis that 
$\ft{\mu}_{P,\Phi^l,L}(\xi) = 0$ for each $1 \leq l \leq N$.
On the other hand, since $\xi\notin V_r^\perp$ there exists a vector 
 $v \in V_r$ such that $\dotprod{\xi}{v} \neq 0$. Hence
\eqref{eq:C1.1} follows from the equality \eqref{eq:D1.20},
and so the proof is complete.
\end{proof}

\subsection{}
Based on \lemref{lemC1.1} we next establish the following result:

\begin{lem}
\label{lemC1.2}
Let $P$ be an element  of the polytope group $\P^d$, and let
$L$ be a lattice in $\R^d$. Assume that
$H_{\Psi}(P,L) = 0$ for every $s$-flag $\Psi$
 $(0 \leq s \leq d-1)$. Then 
$\ft{\chi}_{P}(\xi) = 0$
for all $\xi \in L^* \setminus \{0\}$.
\end{lem}

\begin{proof}
Recall that by a $d$-flag we mean the sequence which consists of just 
one subspace $V_d = \R^d$. If $\Phi$ is a $d$-flag, then the measure
$\mu_{P,\Phi,L}$ is equal to $\chi_P(x) dx$
(that is, an absolutely continuous measure with density $\chi_P$).
Hence, the assertion of \lemref{lemC1.2} can be equivalently stated
by saying  that \eqref{eq:C1.1} holds for all $\xi \in L^* \setminus V_d^\perp$.

The arguments in the proof of \lemref{lemC1.1} are valid also in the 
case when $r=d$. In particular this is true for the second part of the proof,
where it is shown (based on the inductive hypothesis) that \eqref{eq:C1.1} holds for every 
$\xi \in L^* \setminus V_r^\perp$. Notice that the assumption
that $H_{\Psi}(P,L) = 0$ for every $s$-flag $\Psi$ was used in
that part of the proof only for 
$0 \leq s \leq r-1$ (and not for $s = r$). Hence the
proof applies to our present situation as well, and allows us to conclude
that \eqref{eq:C1.1} is indeed true for all $\xi \in L^* \setminus \{0\}$.
\end{proof}

\subsection{}
We can now finish the proof of \thmref{thmC3.1}:

\begin{proof}[Proof of \thmref{thmC3.1}]
Assume first that $P$ tiles at level $k$ by translations with respect to $L$.
Let $D$ be a fundamental parallelepiped  of $L$, then the element
$P' := P - k [D]$ of $\P^d$ tiles at level zero by translations with respect to $L$.
Hence by \propref{propE1.2} we can represent $P'$ as a finite sum 
$P' = \sum ([A_j] - [A'_j])$ where
$A_j$, $A'_j$ are polytopes such that $A'_j$ is obtained from $A_j$ by
translation along a vector from $L$. Since the Hadwiger
functionals $H_\Phi(\,\cdot\, , L)$ are additive and
invariant with respect to translations by vectors from $L$, it follows
that $H_\Phi(P', L) = 0$ for every $r$-flag $\Phi$ $(0 \leq r \leq d-1)$.
It is also easy to check that $H_\Phi(D, L) = 0$ for every such $\Phi$
(in fact, this follows from the proof of \corref{corA2.1}).
Since $P = P' + k[D]$, using additivity again implies that $H_\Phi(P, L) = 0$ for every $\Phi$,
which proves one part of \thmref{thmC3.1}.

To prove the converse part,  suppose that 
$H_{\Phi}(P,L) = 0$ for every $r$-flag $\Phi$ $(0 \leq r \leq d-1)$.
Then we apply \lemref{lemC1.2} and obtain that $\ft{\chi}_{P}(\xi) = 0$
for all $\xi \in L^* \setminus \{0\}$. This implies, by \propref{propD1.4},
that $P$ tiles at some level $k$ by translations with respect to $L$,
and thus concludes the proof of \thmref{thmC3.1}.
\end{proof}

% =========================================================

\section{Remark}

Equidecomposability using translations by vectors from a lattice $L$,
can easily be seen to constitute  an equivalence relation on the set of
all polytopes in $\R^d$. \thmref{thmA2.1} characterizes
 the polytopes which lie in the same equivalence class with respect
to this relation:  \emph{two polytopes  $A$ and $B$ are equidecomposable
using translations by vectors from $L$, if and only if 
$H_{\Phi}(A,L) = H_{\Phi}(B,L)$
for every $r$-flag $\Phi$ $(0 \leq r \leq d)$}.

(We remind the reader that the meaning of the latter condition  in the case when $r=d$
 is  that $A$ and $B$ have the same volume.)

The condition that $H_{\Phi}(P,L) = H_{\Phi}(Q,L)$ for every $\Phi$, in fact makes
sense for any two elements  $P$ and $Q$ of the polytope group $\P^d$.
This  extends the equivalence relation described above to all the elements of 
the polytope group $\P^d$. One may thus wonder whether the equivalence classes with 
respect to this relation admit any geometric description, in a way that generalizes
the notion of equidecomposability of polytopes.

Based on \thmref{thmC3.1}, we can obtain an affirmative answer to this
question. The following version of \thmref{thmA2.1} is true:

\begin{thm}
\label{thmE1.1}
Let $P$ and $Q$ be two elements of the polytope group $\P^d$, and let $L$ be a lattice in $\R^d$. 
Then the following two conditions are equivalent:
\begin{enumerate-math}
	\item \label{thmE1.1.1}
	$P$ and $Q$ can be represented in the form
	\begin{equation}
	\label{eq:E2.2.1}
	P = R + [A_1] + \dots + [A_n],
	\quad
	Q = R +  [B_1] +\dots+ [B_n],
	\end{equation}
	where $R$ is an element of the polytope group $\P^d$, $A_j, B_j$ are polytopes, and
	$B_j$ can be obtained from $A_j$ by translation 
	along some vector $\lam_j \in L$;
	\item \label{thmE1.1.2}
	$H_{\Phi}(P,L) = H_{\Phi}(Q,L)$ for every $r$-flag $\Phi$ $(0 \leq r \leq d)$.
\end{enumerate-math}
\end{thm}

\begin{proof}
Since the Hadwiger functionals $H_\Phi(\,\cdot\, , L)$ are additive and
invariant with respect to translations by vectors from $L$, it is obvious
that \ref{thmE1.1.1} implies \ref{thmE1.1.2}.  In order to prove the converse,
 suppose  that \ref{thmE1.1.2} holds. Then by additivity we have 
$H_\Phi(P-Q, L) = 0$ for every $r$-flag $\Phi$ $(0 \leq r \leq d)$.
Using \thmref{thmC3.1} we deduce that $P-Q$ tiles at some
level $k$ by translations with respect to $L$; and since $P-Q$ has
volume zero, it follows from \propref{propD1.1} that  the tiling level $k$ must be zero.
By \propref{propE1.2} we can therefore represent $P-Q$ in the form $P-Q = \sum_{j=1}^{n}
 ([A_j] - [B_j])$, where $A_j, B_j$ are polytopes such that
$B_j$ is obtained from $A_j$ by translation along a vector from $L$. 
Hence \eqref{eq:E2.2.1} can be satisfied by taking
$R = P - \sum [A_j] = Q - \sum [B_j]$, so condition \ref{thmE1.1.1} is established.
\end{proof}

\begin{remark}
\thmref{thmE1.1} can be viewed as a version of \thmref{thmA2.1} that
applies to all the elements of the polytope  group $\P^d$ (and not just to polytopes).
However, notice that in the special case when $P$ and $Q$ are polytopes, \thmref{thmE1.1}
gives a weaker result than \thmref{thmA2.1}. This is because
the condition \ref{thmE1.1.1} in  \thmref{thmE1.1} does not say that if $P$ and $Q$ are 
polytopes then the element $R$ in \eqref{eq:E2.2.1} can in fact be chosen to be zero.
\end{remark}

% =========================================================

\end{document}